\documentclass[12pt]{amsart}

\usepackage{ucs}

\usepackage{amssymb}
\usepackage{amsthm}
\usepackage{amsmath}
\usepackage{latexsym}
\usepackage[cp1251]{inputenc}
\usepackage{graphicx}
\usepackage{wrapfig}
\usepackage{caption}
\usepackage{subcaption}
\usepackage{indentfirst}
\usepackage[left=2.6cm,right=2.6cm,top=3cm,bottom=3cm,bindingoffset=0cm]{geometry}
\usepackage{enumerate}

\def\O{\mathbf{O}}

\DeclareMathOperator{\aut}{Aut}

\DeclareMathOperator{\cay}{Cay}
\DeclareMathOperator{\cyc}{Cyc}

\DeclareMathOperator{\GL}{GL}

\DeclareMathOperator{\iso}{Iso}

\DeclareMathOperator{\orb}{Orb}

\DeclareMathOperator{\Span}{Span}

\DeclareMathOperator{\sym}{Sym}
\DeclareMathOperator{\rad}{rad}
\DeclareMathOperator{\Sup}{Sup}
\DeclareMathOperator{\DCI}{DCI}
\DeclareMathOperator{\CI}{CI}

\DeclareMathOperator{\NCI}{NCI}
\DeclareMathOperator{\NDCI}{NDCI}
\DeclareMathOperator{\Hol}{Hol}

\newcommand{\overbar}[1]{\mkern 1.5mu\overline{\mkern-1.5mu#1\mkern-1.5mu}\mkern 1.5mu}

\def\tm#1{\item[{\rm (#1)}]}
\makeatletter 
\def\@seccntformat#1{\csname the#1\endcsname. } 
\def\@biblabel#1{#1.}

\makeatother

\title{On CI-property of normal Cayley digraphs over abelian groups}

\author{Grigory Ryabov}
\address{School of Mathematical Sciences, Hebei Key Laboratory of Computational Mathematics and Applications, Hebei Normal University, Shijiazhuang 050024, P. R. China}
\address{Novosibirsk State Technical University, Novosibirsk, Russia}
\email{gric2ryabov@gmail.com}

\thanks{The author was supported by the grant of The Natural Science Foundation of Hebei Province (project No.~A2023205045)}

\date{}

\newtheorem{prop}{Proposition}[section]

\newtheorem{lemm}[prop]{Lemma}
\newtheorem{theo}[prop]{Theorem}

\newtheorem{corl}[prop]{Corollary}
\theoremstyle{definition}

\begin{document}

\vspace{\baselineskip}
\vspace{\baselineskip}

\vspace{\baselineskip}

\vspace{\baselineskip}

\begin{abstract}
A Cayley digraph $\Gamma$ over a finite group $G$ is said to be \emph{CI} if for every Cayley digraph $\Gamma^\prime$ over $G$ isomorphic to $\Gamma$, there is an isomorphism from $\Gamma$ to $\Gamma^\prime$ which is at the same time an automorphism of $G$. In the present paper, we study a $\CI$-property of \emph{normal} Cayley digraphs over abelian groups, i.e. such Cayley digraphs $\Gamma$ that the group $G_r$ of all right translations of $G$ is normal in $\aut(\Gamma)$. At first, we reduce the case of an arbitrary abelian group to the case of an abelian $p$-group. Further, we obtain several results on $\CI$-property of normal Cayley digraphs over abelian $p$-groups. In particular, we prove that every normal Cayley digraph over an abelian $p$-group of order at most $p^5$, where $p$ is an odd prime, is $\CI$. 
\end{abstract}

\maketitle

\section{Introduction}

Let $G$ be a finite group and $S$ a subset of $G$ not containing the identity element of $G$. By a \emph{Cayley digraph} $\Gamma=\cay(G,S)$ over $G$ with \emph{connection set} $S$, we mean a directed graph with vertex set $G$ and arc set $\{(g,sg):~g\in G,~s\in S\}$. If $S$ is inverse-closed, then $\Gamma$ is a \emph{Cayley graph}. Two Cayley digraphs $\Gamma$ and $\Gamma^\prime$ over $G$ are called \emph{Cayley isomorphic} if there exists an isomorphism from $\Gamma$ to $\Gamma^\prime$ which is an automorphism of $G$. Clearly, if Cayley digraphs are Cayley isomorphic, then they are (combinatorially) isomorphic. However, the converse statement does not hold in general. A Cayley digraph $\Gamma$ over $G$ is said to be \emph{CI} if every Cayley digraph over $G$ isomorphic to $\Gamma$ is Cayley isomorphic to $\Gamma$. One of the main motivations to study $\CI$-Cayley digraphs comes from the following observation. If $\Gamma$ is a $\CI$-Cayley digraph over $G$, then an isomorphism between $\Gamma$ and any other Cayley digraph over $G$ can be verified in time $|G|^{O(1)}|\aut(G)|$. In many cases, the latter time is polynomial in $|G|$, i.e. in the number of vertices of $\Gamma$. For more details on $\CI$-Cayley digraphs, we refer the reader to the survey~\cite{Li}. 

A finite group~$G$ is defined to be a \emph{DCI-group} (\emph{CI-group}, respectively) if every Cayley digraph (graph, respectively) over $G$ is $\CI$. The problem of determining of all (D)CI-groups going back to the \'Ad\'am conjecture~\cite{Adam} was posed in the classical paper by Babai and Frankl~\cite{BF}. Strong necessary conditions for a finite group $G$ to be (D)CI can be found in~\cite{DMS,Li,LLP}. In particular, if $G$ is an abelian $\DCI$-group, then every Sylow subgroup of $G$ must be elementary abelian or cyclic of order~$4$. In general, the problem of determining whether a given group $G$ is (D)CI seems to be hard, in particular, because the number of Cayley (di)graphs over $G$ is exponential in $|G|$. A collection of results on $\DCI$- and $\CI$-groups can be found in~\cite[Section~8]{Li}.

Another way to study a $\CI$-property of Cayley digraphs is to study this property for different reasonable classes of them. For instance, the classes of Cayley digraphs of bounded valency and central Cayley digraphs have been investigated in this context (see~\cite[Section~9]{Li} and~\cite{GGRV}, respectively). One more class of Cayley digraphs whose $\CI$-property has been studied is the class of normal Cayley digraphs. It is easy to check that the automorphism group of a Cayley digraph $\Gamma$ over a group $G$ has a regular subgroup $G_r$ consisting of all right translations of $G$ or, in other words, $\aut(\Gamma)\in \Sup(G_r)$, where $\Sup(G_r)$ is the set of all subgroups of the symmetric group $\sym(G)$ containing $G_r$. Given $K_1\leq K_2\leq \sym(G)$, the centralizer and normalizer of $K_1$ in $K_2$ are denoted by $C_{K_2}(K_1)$ and $N_{K_2}(K_1)$, respectively.  The Cayley digraph $\Gamma$ over $G$ is said to be \emph{normal} if $G_r$ is normal in $\aut(\Gamma)$ or, equivalently, 
$$\aut(\Gamma)\leq N_{\sym(G)}(G_r)=\Hol(G)=G_r\rtimes \aut(G).$$
Normal Cayley digraphs have been studied in different areas of the algebraic combinatorics (see, e.g,~\cite{FPW,Praeger,WWX}).

Seemingly, a discussion on a $\CI$-property of normal Cayley digraphs was initiated by Li in~\cite[Section~6.4]{Li}. In this paper, he constructed a normal non-$\CI$ Cayley digraph over a cyclic $2$-group of order at least~$8$ and conjectured that normal non-$\CI$ Cayley digraphs are rare. In~\cite{XFRL}, the following notion was introduced. A finite group $G$ is defined to be an \emph{NDCI-group} (\emph{NCI-group}, respectively) if every normal Cayley digraph (graph, respectively) over $G$ is $\CI$. The problem of determining of all N(D)CI-groups was posed in the same paper. The cyclic, generalized quaternion, and dihedral $\NDCI$- and $\NCI$-groups were classified in~\cite{XFRL},~\cite{XFK}, and~\cite{XFZ}, respectively.

In the present paper, we study a $\CI$-property of normal Cayley digraphs over abelian groups. To do this, we use an approach going back to the following Babai result~\cite{Babai}. A Cayley digraph $\Gamma$ over a group $G$ is $\CI$ if and only if every regular subgroup of $\aut(\Gamma)$ isomorphic to $G$ (further throughout the paper, we will call such subgroups \emph{$G$-regular}) is conjugate in $\aut(\Gamma)$ to $G_r$. Any permutation group $K\in \Sup(G_r)$ having the latter property that every $G$-regular subgroup of $K$ is conjugate in $K$ to $G_r$ is said to be \emph{$G$-transjugate}~\cite{HM}.

Following~\cite{Wi2}, we say that two permutation groups $K_1,K_2\leq \sym(G)$ are called \emph{$2$-equivalent} if $K_1$ and $K_2$ acting on $G^2$ componentwise have the same orbits. The relation of $2$-equivalence is an equivalence relation on the set of all subgroups of $\sym(G)$. Every equivalence class has a unique maximal element with respect to inclusion. Given $K\leq \sym(G)$, this element is called the \emph{$2$-closure} of $K$ and denoted by $K^{(2)}$. The permutation group $K$ is said to be \emph{$2$-closed} if $K^{(2)}=K$. The set of all $2$-closed permutation groups from $\Sup(G_r)$ is denoted by $\Sup_2(G_r)$. An automorphism group of any Cayley digraph over $G$ belongs to $\Sup_2(G_r)$. So if every $K\in \Sup_2(G_r)$ is $G$-transjugate, then an automorphism group of every Cayley digraph over $G$ so is and hence $G$ is a $\DCI$-group. A finite group $G$ is defined to be a \emph{$\CI^{(2)}$-group} if it has the aforesaid property, that is, 
$$K~\text{is}~G\text{-transjugate for every}~K\in \Sup_2(G_r)$$ 
(see~\cite{HM}). Clearly, every $\CI^{(2)}$-group is also $\DCI$, however it is not known whether there exists a $\DCI$-group which is non-$\CI^{(2)}$. All known abelian $\DCI$-groups are also $\CI^{(2)}$ (see~\cite{FK,HM,KM,KMPRS,KR,M1,M2,MS,Ry,Ry2}).

The lattice of all subgroups of $\Hol(G)$ is denoted by $\mathcal{L}(\Hol(G))$. A finite group $G$ is defined to be an \emph{$\NCI^{(2)}$-group} if 
$$K~\text{is}~G\text{-transjugate for every}~K\in \Sup_2(G_r)\cap \mathcal{L}(\Hol(G)).$$ 
If a Cayley digraph is normal, then its automorphism group belongs to $\Sup_2(G_r)\cap \mathcal{L}(\Hol(G))$. Therefore every $\NCI^{(2)}$-group is also $\NDCI$. Given $K\in \Sup_2(G_r)\cap \mathcal{L}(\Hol(G))$ is $G$-transjugate if and only if $G_r$ is a unique $G$-regular subgroup of $K$ because $G_r$ is normal in $K$.

In the present paper, we study abelian $\NCI^{(2)}$-groups and by that find infinite families of abelian groups whose all normal Cayley digraphs are $\CI$. The first main result of the paper provides a necessary and sufficient condition for an abelian group to be $\NCI^{(2)}$.

\begin{theo}\label{main1}
An abelian group is an $\NCI^{(2)}$-group if and only if each of its Sylow subgroups so is.  
\end{theo}

In~\cite{KM}, it was conjectured that a direct product of two $\CI$-groups of coprime orders is also $\CI$. Theorem~\ref{main1} implies that the similar conjecture is true in case of abelian $\NCI^{(2)}$-groups. Observe that Theorem~\ref{main1} does not hold for $\NCI$-groups. Indeed, the cyclic group of order~$8$ and any cyclic group of odd order are $\NCI$ by~\cite[Theorem~1.1]{XFRL}, whereas their direct product is non-$\NCI$ by the same statement.

Theorem~\ref{main1} reduces studying abelian $\NCI^{(2)}$-groups to studying abelian $p$-groups with this property. The second main result provides a complete description of abelian $\NCI^{(2)}$-$2$-groups. Given a positive integer $n$, the cyclic group of order~$n$ is denoted by $C_n$.

\begin{theo}\label{main2}
An abelian $2$-group is an $\NCI^{(2)}$-group if and only if it is isomorphic to $C_4$ or $C_2^k$, where $k\leq 5$.
\end{theo}

Due to~\cite{FK} and Theorem~\ref{main2}, an abelian $2$-group is $\NCI^{(2)}$ if and only if it is $\CI^{(2)}$. The situation is another in case of abelian $p$-groups of odd order. Every abelian $\CI^{(2)}$-$p$-group of odd order must be elementary abelian. However, the class of abelian $\NCI^{(2)}$-$p$-groups of odd order is wider which is demonstrated by the following statement.

\begin{theo}\label{main3}
Let $p$ be an odd prime and $k\geq 1$ an integer. Then $C_{p^k}$ and $C_{p^k}\times C_p$ are $\NCI^{(2)}$-groups.
\end{theo}

As a corollary of Theorems~\ref{main1},~\ref{main2},~and~\ref{main3}, we obtain a complete description of cyclic $\NCI^{(2)}$-groups.

\begin{corl}\label{cycl}
A cyclic group is an $\NCI^{(2)}$-group if and only if its order is not divisible by~$8$. 
\end{corl}

It should be mentioned that~\cite[Theorem~1.1]{XFRL} gives the similar criterion for a cyclic group to be $\NDCI$. Corollary~\ref{cycl} slightly specifies this result. The proof of~\cite[Theorem~1.1]{XFRL} explicitly deals with automorphism groups of Cayley digraphs over cyclic groups, whereas  Corollary~\ref{cycl} is obtained as an immediate consequence from more general results.

Every elementary abelian group of order at most~$p^5$, where $p$ is an odd prime, is $\CI^{(2)}$~\cite{FK}. In case of $\NCI^{(2)}$-groups, the similar result remains true for every abelian group.

\begin{theo}\label{main4}
Let $p$ be an odd prime. Then every abelian $p$-group of order at most~$p^5$ is an $\NCI^{(2)}$-group. 
\end{theo}

Given a positive integer~$n$, the set of all prime divisors of~$n$ is denoted by $\pi(n)$. Given an abelian group $G$ of order~$n$ and $p\in \pi(n)$, the Sylow $p$-subgroup of $G$ is denoted by $G_p$. Theorems~\ref{main1} and~\ref{main4} imply the following corollary.

\begin{corl}
Let $G$ be an abelian group of order~$n$. Suppose that $|G_p|\leq p^5$ for every $p\in \pi(n)$ and $G_2$ is isomorphic to $C_4$ or $C_2^k$, where $k\leq 5$. Then $G$ is an $\NCI^{(2)}$-group.
\end{corl}

We finish the introduction with a brief outline of the paper. A keynote tool used for the proofs of the main results is $S$-rings (Schur rings) (see~\cite{Schur,Wi}). We follow the general idea from~\cite{HM}. Namely, studying normal $\CI$-Cayley digraphs is reduced to studying normal $S$-rings possessing some additional properties. In fact, we show each of the above $S$-rings can be obtained as a tensor product of $p$-$S$-rings whose every basic set has a prime-power size. By that, we reduce the question on $\CI$-property for normal $S$-rings to this question for normal $p$-$S$-rings. Most of the $p$-$S$-rings over abelian $p$-groups from Theorems~\ref{main3} and~\ref{main4} can be obtained as a generalized wreath product of $S$-rings over smaller groups. However, a generalized wreath product of $S$-rings can be normal only in a few special cases (Lemma~\ref{wrnorm}) and hence there are a few normal $p$-$S$-rings over these groups.

Section~$2$ contains a necessary background of $S$-rings, especially, isomorphisms of them. In Section~$3$, we discuss the constructions of tensor and generalized wreath products of $S$-rings. In Section~$4$, we provide several properties of $p$-$S$-rings crucial for the proofs of the main results. The reduction of studying a $\CI$-property for (normal) Cayley digraphs to studying this property for (normal) $S$-rings is discussed in Section~$5$. In the presentation of the material of Sections~$2$-$5$, we follow~\cite{FK,KR}, in general. Theorems~\ref{main1}-\ref{main4} are proved in Sections~$6$-$9$, respectively.

\section{$S$-rings}

\subsection{Basic facts and definitions}

Let $G$ be a finite group and $\mathbb{Z}G$ the integer group ring. The identity element of $G$ is denoted by~$e$. The set of all orbits of $K\leq \sym(G)$ on $G$ is denoted by $\orb(K,G)$. If $X\subseteq G$, then the element $\sum \limits_{x\in X} {x}$ of the group ring $\mathbb{Z}G$ is denoted by~$\underline{X}$. The set $\{x^{-1}:x\in X\}$ is denoted by $X^{-1}$.

A subring  $\mathcal{A}\subseteq \mathbb{Z} G$ is called an \emph{$S$-ring} (a \emph{Schur} ring) over $G$ if there exists a partition $\mathcal{S}=\mathcal{S}(\mathcal{A})$ of~$G$ such that:

$(1)$ $\{e\}\in\mathcal{S}$;

$(2)$  if $X\in\mathcal{S}$, then $X^{-1}\in\mathcal{S}$;

$(3)$ $\mathcal{A}=\Span_{\mathbb{Z}}\{\underline{X}:\ X\in\mathcal{S}\}$.

\noindent The notion of an $S$-ring was introduced in~\cite{Schur} (see also~\cite[Chapter~IV]{Wi}). The elements of $\mathcal{S}$ are called the \emph{basic sets} of $\mathcal{A}$. One can verify that if $X,Y\in \mathcal{S}(\mathcal{A})$, then $XY\in \mathcal{S}(\mathcal{A})$ whenever $|X|=1$ or $|Y|=1$.

A set $T\subseteq G$ is called an \emph{$\mathcal{A}$-set} if $\underline{T}\in \mathcal{A}$ or, equivalently, $T$ is a union of some sets from $\mathcal{S}(\mathcal{A})$. If $T$ is an $\mathcal{A}$-set, then put $\mathcal{S}(\mathcal{A})_T=\{X\in \mathcal{S}(\mathcal{A}):~X\subseteq T\}$. A subgroup $H \leq G$ is called an \emph{$\mathcal{A}$-subgroup} if $H$ is an $\mathcal{A}$-set. One can verify that for every $\mathcal{A}$-set $T$, the groups $\langle T \rangle$ and $\rad(T)=\{g\in G:\ gT=Tg=T\}$ are $\mathcal{A}$-subgroups. 

Let $\{e\}\leq L \unlhd U\leq G$. A section $U/L$ is called an \emph{$\mathcal{A}$-section} if $U$ and $L$ are $\mathcal{A}$-subgroups. If $S=U/L$ is an $\mathcal{A}$-section, then the module
$$\mathcal{A}_S=Span_{\mathbb{Z}}\left\{\underline{X}^{\pi}:~X\in\mathcal{S}(\mathcal{A}),~X\subseteq U\right\},$$
where $\pi:U\rightarrow U/L$ is the canonical epimorphism, is an $S$-ring over $S$.

\begin{lemm}~\cite[Lemma~2.1]{EKP}\label{intersection0}
Let $\mathcal{A}$ be an $S$-ring over a group $G$, $H$ an $\mathcal{A}$-subgroup of $G$, and $X \in \mathcal{S}(\mathcal{A})$. Then the number $|X\cap Hx|$ does not depend on $x\in X$.
\end{lemm}

Given $X\subseteq G$ and $m\in \mathbb{Z}$, put $X^{(m)}=\{x^m:~x\in X\}$.

\begin{lemm}\cite[Theorem~23.9(a)]{Wi}\label{burn}
Let $\mathcal{A}$ be an $S$-ring over an abelian group $G$. Then $X^{(m)}\in \mathcal{S}(\mathcal{A})$ for every $X\in \mathcal{S}(\mathcal{A})$ and every $m\in \mathbb{Z}$ coprime to~$|G|$.
\end{lemm}

If $X,Y\subseteq G$ and $Y=X^{(m)}$ for some $m\in \mathbb{Z}$ coprime to~$|G|$, then we say that $X$ and $Y$ are \emph{rationally conjugate}.

\subsection{Isomorphisms and schurity}

Let $\mathcal{A}$ and $\mathcal{A}^\prime$ be $S$-rings over groups $G$ and $G^\prime$, respectively. A bijection $f$ from $G$ to $G^\prime$ is called a (\emph{combinatorial}) \emph{isomorphism} from $\mathcal{A}$ to $\mathcal{A}^\prime$ if for every $X\in \mathcal{S}(\mathcal{A})$ there is $X^{\prime}\in \mathcal{S}(\mathcal{A}^\prime)$ such that $\{(x^f,y^f)\in G\times G:~yx^{-1}\in X\}=\{(x^\prime,y^\prime):~y^\prime(x^\prime)^{-1}\in X^\prime\}$, i.e. $f$ is an isomorphism from $\cay(G,X)$ to $\cay(G^\prime,X^\prime)$. If there is an isomorphism from $\mathcal{A}$ to $\mathcal{A}^\prime$, we say that $\mathcal{A}$ and $\mathcal{A}^\prime$ are \emph{isomorphic}. Put 
$$\iso(\mathcal{A})=\{f\in \sym(G):~f~\text{is an isomorphism from}~\mathcal{A}~\text{onto an}~\text{$S$-ring over}~G\}.$$

A bijection $f\in\sym(G)$ is defined to be a \emph{(combinatorial) automorphism} of $\mathcal{A}$ if for all $x,y\in G$, the basic sets containing the elements $yx^{-1}$ and $y^f(x^{-1})^f$ coincide. The set of all automorphisms of $\mathcal{A}$ forms a group called the \emph{automorphism group} of $\mathcal{A}$ and denoted by $\aut(\mathcal{A})$. One can see that $\aut(\mathcal{A})\in \Sup_2(G_r)$. The $S$-ring $\mathcal{A}$ is said to be \emph{normal} if $G_r$ is normal in $\aut(\mathcal{A})$ or, equivalently, $\aut(\mathcal{A})\leq \Hol(G)$. The set of all right cosets by an $\mathcal{A}$-subgroup forms an imprimitivity system for $\aut(\mathcal{A})$. If $K\leq \aut(\mathcal{A})$ and $L$ is an $\mathcal{A}$-subgroup of $G$, then the permutation group on $G/L$ induced by $K$ is denoted by $K^{G/L}$.

Let $K\in \Sup(G_r)$. The $\mathbb{Z}$-submodule
$$V(K,G)=\Span_{\mathbb{Z}}\{\underline{X}:~X\in \orb(K_e,G)\}$$
is an $S$-ring over $G$ as it was proved by Schur~\cite{Schur}. An $S$-ring $\mathcal{A}$ over $G$ is called \emph{schurian} if $\mathcal{A}=V(K,G)$ for some $K\in \Sup(G_r)$. One can verify that $K^{(2)}=\aut(\mathcal{A})$ and $\mathcal{A}$ is schurian if and only if $\mathcal{A}=V(\aut(\mathcal{A}),G)$. So the mappings 
$$\mathcal{A}\mapsto \aut(\mathcal{A}),~K\mapsto V(K,G)$$
are mutually inverse one-to-one correspondences between the schurian $S$-rings over $G$ and the groups from $\Sup_2(G_r)$ and vice versa, respectively. Observe that the above mappings are also mutually inverse one-to-one correspondences between the normal schurian $S$-rings over $G$ and the groups from $\Sup_2(G_r)\cap \mathcal{L}(\Hol(G))$ and vice versa, respectively. It is easy to see that $K_1\leq K_2$ if and only if $V(K_1,G)\geq V(K_2,G)$ for all $K_1,K_2\in \Sup(G_r)$.

An $S$-ring $\mathcal{A}$ over a group $G$ is called \emph{cyclotomic} if there exists $M\leq\aut(G)$ such that $\mathcal{S}(\mathcal{A})=\orb(M,G)$. In this case, $\mathcal{A}$ is denoted by $\cyc(M,G)$. Obviously, $\mathcal{A}=V(G_r M,G)$. So every cyclotomic $S$-ring is schurian. The lemma below immediately follows from definitions.

\begin{lemm}\label{cyclchar}
If $\mathcal{A}$ is a cyclotomic $S$-ring over a group $G$, then every characteristic subgroup of $G$ is an $\mathcal{A}$-subgroup.
\end{lemm}

\subsection{Thin radical}

The \emph{thin radical} of $\mathcal{A}$ is defined to be the set 
$$\O_\theta(\mathcal{A})=\{x\in G:~\{x\} \in \mathcal{S}(\mathcal{A})\}.$$ 
It is easy to check that $\O_\theta(\mathcal{A})$ is an $\mathcal{A}$-subgroup.

The following easy lemma follows from the fact that a product of two basic sets of an $S$-ring is also a basic set whenever one of these sets is a singleton.

\begin{lemm}\label{groupring}
Let $\mathcal{A}$ be an $S$-ring over a group $G$ and $X\subseteq \O_\theta(\mathcal{A})$ such that $\langle X \rangle=G$. Then $\mathcal{A}=\mathbb{Z}G$.
\end{lemm}

\begin{lemm}\label{thincoset}
Let $\mathcal{A}$ be an $S$-ring over an abelian group $G$, $V=\O_{\theta}(\mathcal{A})$, and $X\in \mathcal{S}(\mathcal{A})$. Then $Xx^{-1}\cap V$ is a subgroup of $\rad(X)$ not depending on $x\in X$. 
\end{lemm}

\begin{proof}
At first, let us prove that $(Xx^{-1}\cap V)\leq \rad(X)$ for every $x\in X$. Clearly, $e\in Xx^{-1}\cap V$ and we are done if $|Xx^{-1}\cap V|=1$. Further, we assume that $|Xx^{-1}\cap V|\geq 2$. Let $a,b\in Xx^{-1}\cap V$ such that $a\neq b$. Since $a,b\in V$, we obtain $a^{-1}b\in V$ and hence $\{a^{-1}b\}\in \mathcal{S}(\mathcal{A})$. So $Xa^{-1}b\in \mathcal{S}(\mathcal{A})$. One can see that $bx\in Xa^{-1}b\cap X\neq \varnothing$. Therefore $Xa^{-1}b=X$. This implies that $a^{-1}b\in Xx^{-1}\cap V \cap \rad(X)$ for all $a\neq b\in Xx^{-1}\cap V$ and we are done. 

If $x,y\in X$, then $(Xx^{-1}\cap V)\leq \rad(X)$ and $(Xy^{-1}\cap V)\leq \rad(X)$. So $(Xx^{-1}\cap V)\subseteq Xy^{-1}$ and $(Xy^{-1}\cap V)\subseteq Xx^{-1}$, i.e. $Xx^{-1}\cap V=Xy^{-1}\cap V$ which proves that the group $Xx^{-1}\cap V$ does not depend on $x\in X$.
\end{proof}

The next lemma immediately follows from~\cite[Eq.~(4)]{FK}.

\begin{lemm}\label{thincenter}
Let $\mathcal{A}$ be a schurian $S$-ring over an abelian group~$G$ and $V=\O_{\theta}(\mathcal{A})$. Then $V_r=G_r\cap Z(\aut(\mathcal{A}))$.
\end{lemm}

\begin{lemm}\label{intersection}
Let $G$ be an abelian group, $R\neq G_r$ a $G$-regular subgroup of $\sym(G)$, $K=\langle G_r,R\rangle$, $\mathcal{A}=V(K,G)$, and $V=\O_{\theta}(\mathcal{A})$. Then $G_r\cap R=V_r$.
\end{lemm}

\begin{proof}
Since both $G_r$ and $R$ are abelian, we have $G_r\cap R\leq Z(K)$. From~\cite[Proposition~2.1]{FK} it follows that $Z(K)\leq Z(K^{(2)})=Z(\aut(\mathcal{A}))$. Therefore $G_r\cap R\leq Z(\aut(\mathcal{A}))\cap G_r=V_r$, where the latter equality holds by Lemma~\ref{thincenter}. On the other hand, $Z(\aut(\mathcal{A}))\leq C_K(R)=R$, where the latter equality holds because $R$ is abelian and regular. Therefore $V_r=Z(\aut(\mathcal{A}))\cap G_r\leq G_r\cap R$ and hence $G_r\cap R=V_r$.
\end{proof}

\section{Wreath and tensor products}

Let $\mathcal{A}$ be an $S$-ring over $G$. Suppose that $G_1$ and $G_2$ are $\mathcal{A}$-subgroups such that $G=G_1\times G_2$ and $\mathcal{A}_i=\mathcal{A}_{G_i}$, $i\in \{1,2\}$. The $S$-ring $\mathcal{A}$ is defined to be a \emph{tensor product} of $\mathcal{A}_{1}$ and $\mathcal{A}_{2}$ if 
$$\mathcal{S}(\mathcal{A})=\mathcal{S}(\mathcal{A}_{1})\otimes \mathcal{S}(\mathcal{A}_{2})=\{X_1\times X_2:~X_1\in\mathcal{S}(\mathcal{A}_{1}),~X_2\in \mathcal{S}(\mathcal{A}_{2})\}.$$
In this case, we write $\mathcal{A}=\mathcal{A}_{1}\otimes \mathcal{A}_{2}$. One can verify that 

\begin{equation}\label{auttens}
\iso(\mathcal{A}_{1}\otimes \mathcal{A}_{2})=\iso(\mathcal{A}_{1})\times \iso(\mathcal{A}_{2})~\text{and}~\aut(\mathcal{A}_1\otimes \mathcal{A}_2)=\aut(\mathcal{A}_1)\times \aut(\mathcal{A}_2).
\end{equation}

If $X\subseteq G_1\times G_2$, then the projections of $X$ on $G_1$ and $G_2$ are denoted by $X_{G_1}$ and $X_{G_2}$, respectively.

\begin{lemm}\cite[Lemma 2.3]{EKP}\label{proj}
Let $\mathcal{A}$ be an $S$-ring over an abelian group $G=G_1\times G_2$. Suppose that $G_1$ and $G_2$ are $\mathcal{A}$-subgroups. Then 
\begin{enumerate}
\tm{1} $X_{G_i}\in \mathcal{S}(\mathcal{A})$ for all $X\in \mathcal{S}(\mathcal{A})$ and $i=1,2;$

\tm{2} $\mathcal{A} \geq \mathcal{A}_{G_1}\otimes \mathcal{A}_{G_2}$, and the equality is attained whenever $\mathcal{A}_{G_i}=\mathbb{Z}G_i$ for some $i\in \{1,2\}$.
\end{enumerate}
\end{lemm}

The lemma below immediately follows from the second part of Eq.~\eqref{auttens}.

\begin{lemm}\label{tensnorm}
A tensor product of two $S$-rings is normal  if and only if each of the operands so is. 
\end{lemm}

Let $S=U/L$ be an $\mathcal{A}$-section of $G$. The $S$-ring~$\mathcal{A}$ is defined to be the \emph{$S$-wreath product} or \emph{generalized wreath product} of $\mathcal{A}_U$ and $\mathcal{A}_{G/L}$ if $L\trianglelefteq G$ and every basic set $X$ of $\mathcal{A}$ outside~$U$ is a union of some $L$-cosets or, equivalently, $L\leq \rad(X)$ for every $X\in \mathcal{S}(\mathcal{A})_{G\setminus U}$. In this case, we write $\mathcal{A}=\mathcal{A}_U \wr_S \mathcal{A}_{G/L}$. The $S$-wreath product is called \emph{nontrivial} if $L\neq \{e\}$ and $U\neq G$ and \emph{trivial} otherwise. If $L=U$, then the $S$-wreath product coincides with the \emph{wreath product} $\mathcal{A}_L\wr \mathcal{A}_{G/L}$ of $\mathcal{A}_L$ and $\mathcal{A}_{G/L}$. The construction of a generalized wreath product of $S$-rings was introduced in~\cite{EP}. We say that $\mathcal{A}$ is \emph{decomposable} if $\mathcal{A}$ is a nontrivial $S$-wreath product for some $\mathcal{A}$-section $S$ of $G$ and \emph{indecomposable} otherwise.

The proof of the lemma below is a slight modification of the proof of~\cite[Theorem~5.7]{EP2} which is concerned with cyclic groups.

\begin{lemm}\label{wrnorm}
Let $\mathcal{A}$ be a normal $S$-ring over an abelian group $G$. Suppose that $\mathcal{A}$ is a nontrivial $U/L$-wreath product for some $\mathcal{A}$-section $U/L$. Then $L$ and $G/U$ are elementary abelian $2$-groups. 
\end{lemm}

\begin{proof}
Due to~\cite[Theorem~3.4.21]{CP}, for every $h\in L\setminus \{e\}$ there is $f\in \aut(\mathcal{A})$ such that 
$$g^f=
\begin{cases}
g,~g\in U,\\
gh,~g\in G\setminus U.
\end{cases}$$
Then
$$g^2h^2=(gh)^2=(g^f)^2=(g^2)^f=
\begin{cases}
g^2,~g^2\in U,\\
g^2h,~g^2\in G\setminus U
\end{cases}$$
for every $g\in G\setminus U$. Since $h\neq e$, the latter equality implies that $h^2=e$ for every $h\in L\setminus \{e\}$ and $g^2\in U$ for every $g\in G\setminus U$. The first part of this statement yields that $L$ is an elementary abelian $2$-group, whereas the second one yields that $G/U$ so is. 
\end{proof}

The corollary below immediately follows from Lemma~\ref{wrnorm}

\begin{corl}\label{wrnonnorm}
A normal $S$-ring over an abelian group of odd order is indecomposable. 
\end{corl}

\begin{lemm}\label{notsizep}
Under the assumptions of Lemma~\ref{intersection}, suppose that $L$ is an $\mathcal{A}$-subgroup of prime order, $\mathcal{A}_L=\mathbb{Z}L$, and $Lx\in \mathcal{S}(\mathcal{A})$ for some $x\in G$. Then $\mathcal{A}$ is a nontrivial $U/L$-wreath product for some proper $\mathcal{A}$-subgroup~$U$ of $G$.
\end{lemm}

\begin{proof}
We are done if $\langle Y\in\mathcal{S}(\mathcal{A}):~L\nleq \rad(Y) \rangle<G$. Further, we assume that
\begin{equation}\label{gener}
\langle Y\in\mathcal{S}(\mathcal{A}):~L\nleq \rad(Y) \rangle=G.
\end{equation}

Let $Y\in\mathcal{S}(\mathcal{A})$ such that $L\nleq \rad(Y)$. The number $\lambda=|Y\cap Ly|$ does not depend on $y\in Y$ by Lemma~\ref{intersection0}. Assume that $\lambda\geq 2$. Then there are $a_1,a_2\in L$ such that $a_1\neq a_2$ and $a_1y,a_2y\in Y$. This implies that $Ya_2a_1^{-1}\cap Y\neq \varnothing$. Note that $\{a_2a_1^{-1}\}\in\mathcal{S}(\mathcal{A})$ because $\mathcal{A}_L=\mathbb{Z}L$. Therefore $Ya_2a_1^{-1}\in \mathcal{S}(\mathcal{A})$. Together with $Ya_2a_1^{-1}\cap Y\neq \varnothing$, this yields that $Ya_2a_1^{-1}=Y$. So $a_2a_1^{-1}\in \rad(Y)$. As $a_1\neq a_2$, we conclude that $L\leq \rad(Y)$, a contradiction to the assumption on $Y$. Thus, 
\begin{equation}\label{intery}
|Y\cap Ly|=1
\end{equation}
for every $Y\in\mathcal{S}(\mathcal{A})$ such that $L\nleq \rad(Y)$. 

The set $G/L$ is an imprimitivity system for $\aut(\mathcal{A})$ because $L$ is an $\mathcal{A}$-subgoup. Since the kernel $\aut(\mathcal{A})_{G/L}$ of the action of $\aut(\mathcal{A})$ on $G/L$ is normal in $\aut(\mathcal{A})$, one can form the group $\aut(\mathcal{A})_{G/L}G_{r}\leq \aut(\mathcal{A})$. Put $\mathcal{B}=V(\aut(\mathcal{A})_{G/L}G_{r},G)$. One can see that $\mathcal{B}\geq \mathcal{A}$ because $\mathcal{A}$ is schurian and $\aut(\mathcal{A})_{G/L}G_{r}\leq \aut(\mathcal{A})$. By the definition, each basic set of $\mathcal{B}$ is contained in an $L$-coset. Together with Lemma~\ref{groupring} and Eqs.~\eqref{gener} and~\eqref{intery}, this implies that $\mathcal{B}=\mathbb{Z}G$. Therefore $\aut(\mathcal{A})_{G/L}\leq G_r$. The latter inclusion yields that
\begin{equation}\label{kernel}
\aut(\mathcal{A})_{G/L}=L_r.
\end{equation}

By the assumption of the lemma, $Lx\in \mathcal{S}(\mathcal{A})$ for some $x\in G$. Then $(Lx)_r\in\O_{\theta}(\mathcal{A}_{G/L})$. Lemma~\ref{thincenter} implies that $(Lx)_r\in Z(\aut(\mathcal{A})_{G/L})$. Therefore $(Lx)_r\in C_{\aut(\mathcal{A}_{G/L})}(R^{G/L})=R^{G/L}$, where the latter equality holds because $R^{G/L}$ is transitive and abelian and hence regular. Thus, 
$$x_r\in G_r \cap (R \aut(\mathcal{A})_{G/L})=G_r\cap (RL_r)=G_r\cap R=V_r,$$
where the first equality follows from Eq.~\eqref{kernel}, whereas the second and the third ones from Lemma~\ref{intersection}. The condition $x_r\in V_r$ yields that $\{x\}\in \mathcal{S}(\mathcal{A})$ which contradicts to the assumptions of the lemma that $|L|$ is prime and $Lx\in \mathcal{S}(\mathcal{A})$.
\end{proof}

\section{$p$-$S$-rings}

Let $p$ be a prime. An $S$-ring $\mathcal{A}$ over a $p$-group $G$ is called a \emph{$p$-$S$-ring} if every basic set of $\mathcal{A}$ has a $p$-power size. Throughout this section, $p$ is a prime, $n\geq 1$, $G$ is an abelian $p$-group $G$ of order~$p^n$, $\mathcal{A}$ is an $S$-ring over $G$, and $V=\O_{\theta}(\mathcal{A})$. The lemma below collects several known properties of $p$-$S$-rings which we need to prove the main results.

\begin{lemm}\label{psring}
In the above notations, the following statements hold:

\begin{enumerate}

\tm{1} $|V|>1$;

\tm{2} there exists a chain of $\mathcal{A}$-subgroups $\{e\}=G_0<G_1<\cdots<G_r=G$ such that $|G_{i+1}:G_i|=p$ for all $i\in\{0,\ldots,r-1\}$;

\tm{3} if $|G:V|=p$, then $\mathcal{A}=\mathbb{Z}V\wr_{V/L}\mathbb{Z}(G/L)$ for some nontrivial $L\leq V$;

\tm{4} if there is $X\in \mathcal{S}(\mathcal{A})$ such that $|X|=p^{n-1}$, then $\mathcal{A}=\mathcal{A}_H\wr \mathbb{Z}(G/H)$, where $H$ is an $\mathcal{A}$-subgroup of order~$p^{n-1}$;

\tm{5} if $|G|=p$, then $\mathcal{A}=\mathbb{Z}G$; 

\tm{6} if $|G|=p^2$, then $\mathcal{A}=\mathbb{Z}G$ or $\mathcal{A}\cong \mathbb{Z}C_p\wr \mathbb{Z}C_p$. 

\end{enumerate}

\end{lemm}

\begin{proof}
Statements~$(1)$ and~$(2)$ of the lemma are taken from~\cite[Theorem~3.3]{HM}. Statement~$(3)$ was proved in case when $G$ is elementary abelian in~\cite[Proposition~4.3(i)]{KR1}, however the proof remains valid for every abelian $p$-group. Statement~$(4)$ can be found in~\cite[Proposition~3.4(i)]{HM}. Statement~$(5)$ follows Statement~$(1)$, whereas Statement~$(6)$ follows from Statements~$(1)$ and~$(3)$.  
\end{proof}

\begin{lemm}\label{setcoset}
If $\mathcal{A}\neq \mathbb{Z}G$, then there exists $X\in \mathcal{S}(\mathcal{A})_{G\setminus V}$ such that $X$ is a coset by a nontrivial subgroup of $V$. 
\end{lemm}

\begin{proof}
Due to Lemma~\ref{psring}(2), there exists a chain of $\mathcal{A}$-subgroups $\{e\}=G_0<G_1<\cdots<G_r=G$ such that $|G_{i+1}:G_i|=p$ for all $i\in\{0,\ldots,r-1\}$. Let $i_0\in \{2,\ldots,r-1\}$ be the minimal number such that $\mathcal{A}_{G_{i_0}}\neq \mathbb{Z}G_{i_0}$ (such $i_0$ exists because $\mathcal{A}\neq \mathbb{Z}G$). By the definition of $i_0$, we have $\mathcal{A}_{G_{i_0-1}}=\mathbb{Z}G_{i_0-1}$. Therefore $\mathcal{A}_{G_{i_0}}=\mathbb{Z}G_{i_0-1}\wr_{G_{i_0-1}/L}\mathbb{Z}(G_{i_0}/L)$ for some nontrivial $L\leq G_{i_0-1}\leq V$ by Lemma~\ref{psring}(3). This implies that every set from $\mathcal{S}(\mathcal{A})_{G_{i_0}\setminus G_{i_0-1}}$ is an $L$-coset and we are done.
\end{proof}

\begin{lemm}\label{interrad}
Let $X\in \mathcal{S}(\mathcal{A})$ such that $\langle X \rangle=G$. Then 
$$|\rad(X)\cap V|\geq p|V||X|/|G|.$$
Moreover, if the equality is attained and $p|V||X|/|G|>1$, then $\mathcal{A}$ is the nontrivial $H/L$-wreath product, where $L=\rad(X)\cap V$ and $H$ is an $\mathcal{A}$-subgroup of index~$p$.
\end{lemm}

\begin{proof}
In view of Lemma~\ref{burn}, $X^{(m)}g\in \mathcal{S}(\mathcal{A})$ for all $g\in V$ and $m\in\{1,\ldots,p-1\}$. Let $\mathcal{X}$ be the set of all basic sets of the above form and $\mathcal{X}^{\cup}$ the union of them. Due to Lemma~\ref{psring}(2), there exists an $\mathcal{A}$-subgroup $H$ of index~$p$. Since $\langle X \rangle=G$, we have $\langle Y \rangle=G$ for every $Y\in \mathcal{X}$ and consequently
\begin{equation}\label{inclus}
\mathcal{X}^{\cup}\subseteq G\setminus H.
\end{equation}

Lemma~\ref{psring}(5) implies that $\mathcal{A}_{G/H}=\mathbb{Z}(G/H)$. So $X^{(m_1)}g_1$ and $X^{(m_2)}g_2$, where  $m_1\neq m_2\in \{1,\ldots,p-1\}$ and $g_1,g_2\in V$, are subsets of distinct $H$-cosets and hence $X^{m_1}g_1\cap X^{m_2}g_2=\varnothing$. Let $L=\rad(X)\cap V$. Clearly, 
\begin{equation}\label{radicals}
V\cap \rad(Y)=L 
\end{equation}
for every $Y\in \mathcal{X}$. Therefore $X^{(m_1)}g_1=X^{(m_1)}g_2$ if and only if $Lg_1=Lg_2$. The discussion of this paragraph yields that
\begin{equation}\label{chain0}
|\mathcal{X}^{\cup}|\geq (p-1)|X||V|/|L|.
\end{equation}

From Eqs.~\eqref{inclus} and~\eqref{chain0} it follows that
\begin{equation}\label{chain}
(p-1)|X||V|/|L|\leq |\mathcal{X}^{\cup}|\leq |G|-|H|=(p-1)|G|/p.
\end{equation}
Eq.~\eqref{chain} implies that
$$|L|\geq p|V||X|/|G|$$
as required. If the equality is attained in the latter inequality and $p|V||X|/|G|>1$, then $\mathcal{X}^{\cup}=G\setminus H$ by Eq.~\eqref{chain} and $|L|>1$. Therefore $\mathcal{A}$ is the nontrivial $H/L$-wreath product by Eq.~\eqref{radicals} and we are done.
\end{proof}

\begin{lemm}\label{sizep}
Let $\mathcal{A}=\cyc(M,G)$ for some $p$-group $M\leq \aut(G)$ and $X\in \mathcal{S}(\mathcal{A})$ such that $|X|=p$. Then there is $Y\in \mathcal{S}(\mathcal{A})$ such that $Y$ is a coset by a subgroup of $V$ of order~$p$.
\end{lemm}

\begin{proof}
It suffices to prove that $\mathcal{A}_{\langle X \rangle}$ has a basic set which is a coset by a subgroup of $V$ of order~$p$. So we may assume further that $\langle X \rangle=G$. Let $x\in X$ and $\mathcal{B}=\cyc(M_x,G)$. Note that $M_x<M$ because $|X|=p$. Together with $\{x\}\in \mathcal{S}(\mathcal{B})$, this implies that $\mathcal{B}>\mathcal{A}$. Therefore $X$ is a $\mathcal{B}$-set. Clearly, $M_x$ is a $p$-group and hence $\mathcal{B}$ is a $p$-$S$-ring. In particular, every  set from $\mathcal{S}(\mathcal{B})_X$ has a $p$-power size. Together with $|X|=p$ and $\{x\}\in \mathcal{S}(\mathcal{B})$, this yields that each set from $\mathcal{S}(\mathcal{B})_X$ is a singleton. Now Lemma~\ref{groupring} and $\langle X \rangle=G$ imply that $\mathcal{B}=\mathbb{Z}G$. Therefore $|M_x|=1$ and hence $|M|=p$. Thus, every basic set from $\mathcal{S}(\mathcal{A})_{G\setminus V}$ is of size~$p$ and we are done by Lemma~\ref{setcoset}.
\end{proof}

\begin{lemm}\label{genrad}
Suppose that $\mathcal{A}=V(K,G)$, where $K=\langle G_r,R\rangle$ for some $G$-regular subgroup $R\neq G_r$ of $\sym(G)$, and $X\in \mathcal{S}(\mathcal{A})$ generates $G$. Then $|\rad(X)\cap V|>1$. 
\end{lemm}

\begin{proof}
From Lemma~\ref{setcoset} it follows that there exists $Y\in \mathcal{S}(\mathcal{A})$ which is an $H$-coset for some nontrivial subgroup $H$ of $V$. Let $L$ be the subgroup (possibly, trivial) of $H$ of index~$p$. Clearly, $H/L\leq V/L$, $|H/L|=p$, and $Y/L\in \mathcal{S}(\mathcal{A}_{G/L})$ is an $H/L$-coset.

It is easy to see that $\mathcal{A}_{G/L}=V(K^{G/L},G)$ and $K^{G/L}=\langle G_r^{G/L},R^{G/L}\rangle=\langle (G/L)_r,R^{G/L}\rangle$. The group $R^{G/L}$ induced by $R$ on $G/L$ is transitive and abelian and hence regular. If $R^{G/L}=(G/L)_r$, then $|L|>1$ and $\mathcal{A}_{G/L}=\mathbb{Z}(G/L)$. So every basic set of $\mathcal{A}$ is contained in an $L$-coset. If $|X\cap Lx|>1$ for some $x\in X$, then $|\rad(X)\cap V|>1$ by Lemma~\ref{thincoset} and we are done. Otherwise, $X$ is a singleton and Lemma~\ref{groupring} yields that $\mathcal{A}=\mathbb{Z}G$ which contradicts to $R\neq G_r$. Thus, we may assume that 
$$R^{G/L}\neq (G/L)_r.$$

The above paragraphs imply that $\mathcal{A}_{G/L}$ and $\overbar{H}=H/L$ satisfy all the conditions of Lemma~\ref{notsizep}. So $\mathcal{A}_{G/L}$ is a $\overbar{U}/\overbar{H}$-wreath product for some proper $\mathcal{A}_{G/L}$-subgroup $\overbar{U}$ of $G/L$. Since $\langle \overbar{X}\rangle=G/L$, we conclude that $\overbar{X}\nsubseteq \overbar{U}$ and consequently $\overbar{H}\leq \rad(\overbar{X})$, where $\overbar{X}$ is the image of $X$ under the canonical epimorphism from $G$ to $G/L$. As $L<H\leq V$, we obtain $|X\cap Vx|>1$ for some $x\in X$. Thus, $|\rad(X)\cap V|>1$ by Lemma~\ref{thincoset} as required.  
\end{proof}

\begin{lemm}\label{faith}
Suppose that $\mathcal{A}$ is normal, $\mathcal{A}=\cyc(M,G)$ for some abelian $p$-subgroup $M$ of $\aut(G)$, and $|G:\langle X \rangle|\leq p$ for some $X\in \mathcal{S}(\mathcal{A})$. Then $X$ is a faithful regular orbit of $M$. 
\end{lemm}

\begin{proof}
The group $M^X$ induced by $M$ on $X$ is transitive and abelian. So $M^X$ is regular and hence $X$ is a regular orbit of $M$. Let us prove that $X$ is also faithful. Let $U=\langle X \rangle$ and $M_X$ the kernel of the action of $M$ on $X$. If $U=G$, then $M_X$ is trivial and hence $X$ is a faithful orbit of $M$.  Suppose that $|G:U|=p$. Let $\mathcal{B}=\cyc(M_X,G)$. Observe that $\mathcal{B}\geq \mathcal{A}$ because $M_X\leq M$. Since $M$ is a $p$-group, $M_X$ so is. This implies that $\mathcal{B}$ is a $p$-$S$-ring. Clearly, $M_X$ acts trivially on $X$ and hence $U \leq \O_{\theta}(\mathcal{B})$. If $\O_{\theta}(\mathcal{B})=G$, then $M_X$ is trivial and consequently $X$ is a faithful orbit of $M$. Otherwise, $\O_{\theta}(\mathcal{B})=U$. Therefore 
$$\mathcal{B}=\mathbb{Z}U \wr_{U/L}\mathbb{Z}(G/L)$$ 
for some nontrivial $L\leq U$ by Lemma~\ref{psring}(3). Since $U$ is an $\mathcal{A}$-subgroup as well as $\mathcal{B}$-subgroup, we conclude that $\mathcal{A}$ is the nontrivial $U/L_1$-wreath product, where $L_1$ is the least $\mathcal{A}$-subgroup containing $L$, a contradiction to Corollary~\ref{wrnonnorm}. 
\end{proof}

\section{$\CI$-$S$-rings}

Let $\mathcal{A}$ be an $S$-ring over a group $G$. It is easy to check that $\aut(\mathcal{A})\aut(G)\subseteq \iso(\mathcal{A})$. However, the converse inclusion does not hold in general. The $S$-ring $\mathcal{A}$ is said to be \emph{CI} if 
$$\iso(\mathcal{A})=\aut(\mathcal{A})\aut(G).$$ 
The notion of a $\CI$-$S$-ring was introduced in~\cite{HM}.

Let $K\in \Sup_2(G_r)$ and $\mathcal{A}=V(K,G)$. Clearly, $K=\aut(\mathcal{A})$. Due to~\cite[Theorem~2.6]{HM}, the $S$-ring $\mathcal{A}$ is $\CI$ if and only if $K$ is $G$-transjugate. So $G$ is a $\CI^{(2)}$-group if and only if every schurian $S$-ring over $G$ is $\CI$. Observe that $K\in \Sup_2(G_r)\cap \mathcal{L}(\Hol(G))$ if and only if $\mathcal{A}$ is normal. Therefore we obtain the next statement.

\begin{lemm}\label{ncicrit0}
A group $G$ is an $\NCI^{(2)}$-group if and only if every normal schurian $S$-ring over $G$ is $\CI$.
\end{lemm}

Let $K_1,K_2\in \Sup(G_r)$ such that $K_1 \leq K_2$. The group $K_1$ is called a \emph{$G$-complete subgroup} of $K_2$ if every $G$-regular subgroup of $K_2$ is conjugate in $K_2$ to some $G$-regular subgroup of $K_1$ (see \cite[Definition~2]{HM}). In this case, we write $K_1 \preceq_G K_2$. The relation $\preceq_G$ is a partial order on $\Sup(G_r)$. The set of the minimal elements of $\Sup_2(G_r)$ with respect to $\preceq_G$ is denoted by $\Sup_2^{\min}(G_r)$. One can see that if $K_1 \preceq_G K_2$ and $K_1$ is $G$-transjugate, then $K_2$ so is. Therefore $G$ is a $\CI^{(2)}$-group if and only if every schurian $S$-ring $\mathcal{A}$ such that $\aut(\mathcal{A})\in \Sup_2^{\min}(G_r)$ is $\CI$. Clearly, if $K_1 \preceq_G K_2$ and $K_2\leq \Hol(G)$, then $K_1\leq \Hol(G)$. Thus, we obtain the following criterion of an $\NCI^{(2)}$-property for groups which is a refinement of Lemma~\ref{ncicrit0}.

\begin{lemm}\label{ncicrit}
A group $G$ is an $\NCI^{(2)}$-group if and only if every normal schurian $S$-ring $\mathcal{A}$ such that $\aut(\mathcal{A})\in \Sup_2^{\min}(G_r)$ is a $\CI$-$S$-ring.
\end{lemm}

The lemma below provides an important property of schurian $S$-rings whose automorphism groups belong to $\Sup_2^{\min}(G_r)$. 

\begin{lemm}\cite[Lemma~5.2]{KM}\label{minpring}
Let $G$ be an abelian group and $\mathcal{A}$ a schurian $S$-ring over $G$ such that $\aut(\mathcal{A})\in \Sup_2^{\min}(G_r)$. Suppose that $H$ is an $\mathcal{A}$-subgroup of $G$ such that $G/H$ is a $p$-group for some prime $p$. Then $\mathcal{A}_{G/H}$ is a $p$-$S$-ring.
\end{lemm}

Further, we give a necessary condition of a non-$\NCI^{(2)}$-property for a group that will be used in the proofs of Theorems~\ref{main3} and~\ref{main4}.

\begin{lemm}\label{noncinorm}
Let $G$ be a non-$\NCI^{(2)}$-group. Then there is a $G$-regular subgroup $R\neq G_r$ of $\Hol(G)$ such that $\mathcal{A}=V(G_rR,G)$ is a normal non-$\CI$-$S$-ring. If in addition, $G$ is a $p$-group, then $\mathcal{A}$ is a $p$-$S$-ring.  
\end{lemm}

\begin{proof}
Since $G$ is a non-$\NCI^{(2)}$-group, there is $K\in \Sup_2(G_r)\cap \mathcal{L}(\Hol(G))$ such that $K$ is not $G$-transjugate. This implies that there is a $G$-regular subgroup $R\neq G_r$ of $K$. One can form the group $G_rR$ because $R\leq K\leq \Hol(G)=N_{\sym(G)}(G_r)$. Due to the definitions of $K$ and $R$, we have
\begin{equation}\label{2close}
(G_rR)^{(2)}\leq K^{(2)}=K\leq \Hol(G).
\end{equation}
As $G_r$ and $R$ are not conjugate in $\Hol(G)$, they are not conjugate in $(G_rR)^{(2)}$. Therefore the group $(G_rR)^{(2)}$ is not $G$-tranjugate and consequently $\mathcal{A}=V(G_rR,G)$ is a non-$\CI$-$S$-ring. Clearly, $\aut(\mathcal{A})=(G_rR)^{(2)}$ and hence the normality of $\mathcal{A}$ follows from Eq.~\eqref{2close}. If $G$ is a $p$-group, then $G_rR$ so is. This implies that every basic set of $\mathcal{A}$ being an orbit of $G_rR$ has a $p$-power size, i.e $\mathcal{A}$ is a $p$-$S$-ring. 
\end{proof}

In the end of the section, we give a criterion of a $\CI$-property for a tensor product of $S$-rings and a corollary of it which allows to construct non-$\NCI^{(2)}$-groups.

\begin{lemm}\label{tensci}
A tensor product of two $S$-rings is $\CI$ if and only if each of the operands so is. 
\end{lemm}

\begin{proof}
Let $G=G_1\times G_2$ and $\mathcal{A}$ an $S$-ring over $G$. Suppose that $G_1$ and $G_2$ are $\mathcal{A}$-subgroups, $\mathcal{A}_i=\mathcal{A}_{G_i}$, $i\in \{1,2\}$, and $\mathcal{A}=\mathcal{A}_1\otimes \mathcal{A}_2$. If $\mathcal{A}_1$ and $\mathcal{A}_2$ are $\CI$, then
$$\aut(\mathcal{A})\aut(G)\leq\iso(\mathcal{A})=\iso(\mathcal{A}_1)\times \iso(\mathcal{A}_2)=$$
$$=(\aut(\mathcal{A}_1)\aut(G_1))\times (\aut(\mathcal{A}_2)\aut(G_2))=(\aut(\mathcal{A}_1)\times \aut(\mathcal{A}_2))(\aut(G_1)\times \aut(G_2))=$$
$$=\aut(\mathcal{A})(\aut(G_1)\times \aut(G_2))\leq \aut(\mathcal{A})\aut(G),$$
where the first and fourth equalities follow from the first and second parts of Eq.~\eqref{auttens}, respectively, and the second equality follows from the assumption that $\mathcal{A}_1$ and $\mathcal{A}_2$ are $\CI$. Therefore $\iso(\mathcal{A})=\aut(\mathcal{A})\aut(G)$, i.e. $\mathcal{A}$ is $\CI$. 

Conversely, suppose that $\mathcal{A}$ is $\CI$. Let $R_1$ be a $G_1$-regular subgroup of $\aut(\mathcal{A}_1)$. Then $R_1\times (G_2)_r$ is a $G$-regular subgroup of $\aut(\mathcal{A})=\aut(\mathcal{A}_1)\times \aut(\mathcal{A}_2)$. Since $\mathcal{A}$ is $\CI$, there exist $f_1\in \aut(\mathcal{A}_1)$ and $f_2\in \aut(\mathcal{A}_2)$ such that 
$$R_1^{f_1}\times (G_2)^{f_2}_r=(R_1\times (G_2)_r)^{f_1\times f_2}=G_r=(G_1)_r\times (G_2)_r.$$
Therefore $R_1^{f_1}=(G_1)_r$. Thus, $\aut(\mathcal{A}_1)$ is $(G_1)_r$-transjugate and hence $\mathcal{A}_1$ is $\CI$. The $\CI$-property of $\mathcal{A}_2$ can be checked in the same way.
\end{proof}

It should be mentioned that the ``if'' part of the above lemma in case of an abelian group follows from more general theory~\cite[Proposition~3.2,~Theorem~4.1]{KM}.

\begin{corl}\label{prodnonci}
If $G$ is a non-$\NCI^{(2)}$-group, then $G\times H$ is a non-$\NCI^{(2)}$-group for every group $H$.
\end{corl}

\begin{proof}
Since $G$ is non-$\NCI^{(2)}$, there is a normal non-$\CI$-$S$-ring $\mathcal{A}$ over $G$ by Lemma~\ref{ncicrit}. One can form the $S$-ring $\mathcal{B}=\mathcal{A}\otimes \mathbb{Z}H$ over $G\times H$. Note that $\mathcal{B}$ is normal by Lemma~\ref{tensnorm} and non-$\CI$ by Lemma~\ref{tensci}. Therefore $G\times H$ is non-$\NCI^{(2)}$.
\end{proof}

\section{Proof of Theorem~\ref{main1}}

We start with an auxiliary lemma.

\begin{lemm}\label{sylowtens}
Let $\mathcal{A}$ be an $S$-ring over an abelian group $G$ of order~$n$. Suppose that $G_p$ is an $\mathcal{A}$-subgroup and $\mathcal{A}_{G_p}$ is a $p$-$S$-ring for every $p\in \pi(n)$. Then 
\begin{equation}\label{bigtens}  
\mathcal{A}=\bigotimes \limits_{p\in \pi(n)} \mathcal{A}_{G_p}.
\end{equation}
\end{lemm} 

\begin{proof}
We proceed by the induction on~$k=|\pi(n)|$. If $k=1$, then the lemma is trivial. Let $k\geq 2$, $p\in \pi(n)$, and $P=G_p$ and $H$ the Sylow $p$-subgroup and Hall $p^\prime$-subgroup of $G$, respectively. Then $P$ and $H$ are nontrivial $\mathcal{A}$-subgroups of $G$ and $G=P\times H$. By the induction hypothesis, Eq.~\eqref{bigtens} holds for $\mathcal{A}_H$. This implies that every basic set of $\mathcal{A}_H$ is of $p^\prime$-size.

Let $X\in \mathcal{S}(\mathcal{A})_{G\setminus (P\cup H)}$. To prove the lemma, it is enough to show that $X=X_P\times X_H$. By Lemma~\ref{intersection0}, the numbers $\lambda=|X\cap Hx|$ and $\mu=|X\cap Px|$ do not depend on $x\in X$. Therefore 
\begin{equation}\label{size}
|X|=\lambda|X_P|=\mu|X_H|.
\end{equation}
Due to Lemma~\ref{proj}(1), the sets $X_P$ and $X_H$ are basic sets of $\mathcal{A}_P$ and $\mathcal{A}_H$, respectively. So $|X_P|$ and $|X_H|$ are $p$- and $p^\prime$-numbers, respectively. Eq.~\eqref{size} yields that $\lambda$ is divisible by~$|X_H|$, whereas $\mu$ is divisible by~$|X_P|$. Consequently, $|X|\geq |X_P||X_H|$. On the other hand, obviously, $|X|\leq |X_P||X_H|$ and hence $|X|=|X_P||X_H|$. Thus, $X=X_P\times X_H$ as required.
\end{proof}

The ``only if'' part immediately follows from Corollary~\ref{prodnonci}. Further, we prove the ``if'' part. Let $G$ be an abelian group of order~$n$ whose all Sylow subgroups are $\NCI^{(2)}$. Due to Lemma~\ref{ncicrit}, it is enough to show that $\mathcal{A}=V(K,G)$ is a $\CI$-$S$-ring for every $K\in \Sup_2^{\min}(G_r)\cap \mathcal{L}(\Hol(G))$. Since $K\leq \Hol(G)$, the $S$-ring $\mathcal{A}$ is cyclotomic. By Lemma~\ref{cyclchar}, the Sylow $p$-subgroup~$P$ and the Hall $p^\prime$-subgroup $H$ of $G$ are $\mathcal{A}$-subgroups for every $p\in \pi(n)$. Lemma~\ref{minpring} implies that $\mathcal{A}_P\cong \mathcal{A}_{G/H}$ is a $p$-$S$-ring. Therefore $\mathcal{A}$ satisfies the conditions of Lemma~\ref{sylowtens} and hence Eq.~\eqref{bigtens} holds for $\mathcal{A}$. Now the required follows from the condition of the theorem, Lemma~\ref{tensnorm}, and Lemma~\ref{tensci}.

\section{Proof of Theorem~\ref{main2}}

We start the proof with two keynote lemmas. 

\begin{lemm}\label{small2groups}
The groups $C_{2^k}$, $k\geq 3$, $C_4\times C_2$, and $C_4\times C_4$ are not $\NCI^{(2)}$-groups.
\end{lemm}

\begin{proof}
From~\cite[Example~6.10]{Li} it follows that $C_{2^k}$, $k\geq 3$, is a non-$\NDCI$-group and hence it is a non-$\NCI^{(2)}$-group. Further, we are going to construct non-$\CI$ normal schurian $S$-rings over $C_4 \times C_2$ and $C_4\times C_4$. In view of Lemma~\ref{ncicrit0}, this will imply that $C_4 \times C_2$ and $C_4\times C_4$ are non-$\NCI^{(2)}$.

Let $A\cong C_4$, $B\cong C_2$ or $B\cong C_4$, $A_0$ and $B_0$ the subgroups of order~$2$ of $A$ and $B$, respectively, and $G=A\times B$. Denote generators of $A$, $B$, and $B_0$ by $a$, $b$, and $b_0$, respectively. Let $\sigma\in \aut(G)$ such that $a^\sigma=ab_0$ and $b^\sigma=b$. Put $M=\langle \sigma \rangle$ and $\mathcal{A}=\cyc(M,G)$. A computer calculation using~\cite{GAP} yields that $\aut(\mathcal{A})=G_r\rtimes M$ and hence $\mathcal{A}$ is normal. Put $R=\langle a_r\sigma,b_r\rangle$. Clearly, $R\leq \aut(\mathcal{A})$. It can be verified explicitly that $|a_r\sigma|=4$ and $a_r\sigma$ and $b_r$ commute. So $R\cong G$. The transversal $\{e,ab_0,a_0b_0,aa_0\}$ for $B$ in $G$ is an orbit of $\langle a_r\sigma \rangle$. This implies that $R$ is transitive on $G$. Therefore $R$ is a regular subgroup of $\aut(\mathcal{A})$ isomorphic to $G$ and $R\neq G_r$. Thus, $\mathcal{A}$ is non-$\CI$ as desired.
\end{proof}

\begin{lemm}\label{rank62groups}
The groups $C_2^k$, $k\geq 6$, are not $\NCI^{(2)}$-groups.
\end{lemm}

\begin{proof}
In view of Corollary~\ref{prodnonci}, it is enough to prove that the group $C_2^6$ is not an $\NCI^{(2)}$-group. We are going to show that the $S$-ring of a non-$\CI$-Cayley graph constructed in~\cite{Now} is a normal non-$\CI$-$S$-ring over $C_2^6$. Let $G\cong C_2^6$ and $\{a_1,a_2,a_3,a_4,a_5,a_6\}$ a generating set of $G$. Let $\overbar{a}=(a_1,a_2,a_3,a_4,a_5,a_6)$ and $M\leq \aut(G)$ consisting of all automorphisms of $G$ of the form $\overbar{a}\mapsto W\overbar{a}$, where $W$ runs over the group 
$$\left\{ \left(\begin{smallmatrix}
1 & 0 & 0 & 0 & \alpha & \beta\\
0 & 1 & 0 & \alpha & 0 & \gamma\\
0 & 0 & 1 & \beta & \gamma & 0\\
0 & 0 & 0 & 1 & 0 & 0\\
0 & 0 & 0 & 0 & 1 & 0\\
0 & 0 & 0 & 0 & 0 & 1
\end{smallmatrix}\right)|~\alpha,\beta,\gamma\in \mathbb{F}_p\right\}\leq \GL_2(6).$$
Put $\mathcal{A}=\cyc(M,G)$. A computer calculation using~\cite{GAP} based on the Weisfeiler-Leman algorithm~\cite{WeisL} implies that $\mathcal{A}$ is the smallest (with respect to inclusion) $S$-ring over $G$ for which a connection set of the non-$\CI$-Cayley graph~$X$ from~\cite{Now} is an $\mathcal{A}$-set. Therefore $\aut(\mathcal{A})=\aut(X)$ (see, e.g.,~\cite[Corollary~2.6.6]{CP}). So by~\cite{Now}, the group $\aut(\mathcal{A})$ has at least two non-conjugate regular subgroups isomorphic to $G$. Thus, $\mathcal{A}$ is non-$\CI$.

To prove the lemma, it remains to verify that $\mathcal{A}$ is normal. A computer calculation using~\cite{GAP} implies that $\aut(\mathcal{A})=G_r\rtimes M$ and we are done. 
\end{proof}

Let $G$ be an abelian $\NCI^{(2)}$-$2$-group. Then $G$ is a direct product of cyclic $2$-groups. If one of these cyclic groups is of order at least~$8$, then $G$ is non-$\NCI^{(2)}$ by Corollary~\ref{prodnonci} and Lemma~\ref{small2groups}. Therefore $G$ is a direct product of cyclic groups of orders~$2$ or~$4$. If $G$ has a cyclic subgroup of order~$4$, then $G\cong C_4$ because otherwise $G=G_1\times G_2$ for some group $G_2$, where $G_1\cong C_4\times C_2$ or $G_1\cong C_4\times C_4$, and hence $G$ is non-$\NCI^{(2)}$ by Corollary~\ref{prodnonci} and Lemma~\ref{small2groups}. If $G$ has no a cyclic subgroup of order~$4$, then $G\cong C_2^k$ for some $k\geq 1$. Lemma~\ref{rank62groups} implies that $k\leq 5$. Thus, if $G$ is an abelian $\NCI^{(2)}$-$2$-group, then $G\cong C_4$ or $G\cong C_2^k$ for some $k\leq 5$. The latter groups are $\CI^{(2)}$-groups (see~\cite{FK,M1}) and hence $\NCI^{(2)}$.

\section{Proof of Theorem~\ref{main3}}

Let $p$ be an odd prime, $k\geq 1$ an integer, and $G$ is isomorphic to one of the groups $C_{p^k}$ and $C_{p^k}\times C_p$. Assume that $G$ is a non-$\NCI^{(2)}$-group. Then there is a $G$-regular subgroup $R\neq G_r$ of $\Hol(G)$ such that $\mathcal{A}=V(K,G)$ is a normal non-$\CI$-$p$-$S$-ring, where $K=G_rR$, by Lemma~\ref{noncinorm}. Clearly, $\mathcal{A}=\cyc(K_e,G)$. Put $V=\O_{\theta}(\mathcal{A})$.

Let $G\cong C_{p^k}$. Then from~\cite[Theorem~4.1]{EP3} it follows that $|\rad(\mathcal{A})|=1$, where $\rad(\mathcal{A})$ is a subgroup generated by the radicals of all basic sets of $\mathcal{A}$ containing a generator of $G$, or $\mathcal{A}$ is decomposable. In the former case, $|K_e|\leq p-1$ by~\cite[Lemma~5.2]{EP2}, a contradiction to the fact that $K$ and hence $K_e$ are $p$-groups. In the latter one, we obtain a contradiction to Corollary~\ref{wrnonnorm}.

Now let $G\cong C_{p^k}\times C_p$. The subgroup $E$ of $G$ isomorphic to $C_p\times C_p$ is characteristic. So $E$ is an $\mathcal{A}$-subgroup by Lemma~\ref{cyclchar}. If $\mathcal{A}_E\cong \mathbb{Z}C_p\wr \mathbb{Z}C_p$, then $\mathcal{A}$ has a basic set which is a coset by a subgroup of $V$ of order~$p$. Therefore $\mathcal{A}$ is decomposable by Lemma~\ref{interrad}, a contradiction to Corollary~\ref{wrnonnorm}. Thus,
\begin{equation}\label{thinsubset}
E\leq V.
\end{equation}

Suppose that $X\in \mathcal{S}(\mathcal{A})$ contains an element of order~$p^k$ and $U=\langle X \rangle$. Then $U\cong C_{p^k}$ or $U=G$. In the former case, Eq.~\eqref{thinsubset} implies that there is an $\mathcal{A}$-subgroup $L\leq E$ of order~$p$ such that $G=U\times L$. So $\mathcal{A}=\mathcal{A}_U\otimes \mathbb{Z}L$ by Lemma~\ref{proj}(2). Since $U$ is cyclic, $\mathcal{A}_U$ is $\CI$ by the first part of the theorem and hence $\mathcal{A}$ is $\CI$ by Lemma~\ref{tensci}, a contradiction to the fact that $\mathcal{A}$ is non-$\CI$. 

In the latter case when $U=G$, we have
$$\bigcup X^{(m)}=G\setminus H$$
for some $H<G$, where $m$ runs over all integers coprime to~$p$. Observe that $X^{(m)}\in \mathcal{S}(\mathcal{A})$ and $\rad(X^{(m)})=\rad(X)$ for every $m$ coprime to~$p$. Indeed, the first part follows from Lemma~\ref{burn}, whereas the second one is clear. Therefore $H$ is an $\mathcal{A}$-subgroup and $\mathcal{A}$ is the $H/L$-wreath product, where $L=\rad(X)$. Lemma~\ref{genrad} yields that $|L|>1$. Thus, $\mathcal{A}$ is decomposable, a contradiction to Corollary~\ref{wrnonnorm}.

\section{Proof of Theorem~\ref{main4}}
 
Let $p$ be an odd prime, $n\geq 1$, and $G$ an abelian group of order~$p^n$. Assume that $G$ is a non-$\NCI^{(2)}$-group. Then there is a $G$-regular subgroup $R\neq G_r$ of $\Hol(G)$ such that $\mathcal{A}=V(K,G)$, where $K=G_rR$, is a normal non-$\CI$-$p$-$S$-ring by Lemma~\ref{noncinorm}. Clearly, $K\leq \Hol(G)$ and hence $\mathcal{A}=\cyc(M,G)$, where $M=K_e$. One can see that $\mathcal{A}\neq\mathbb{Z}G$ because $R\neq G_r$. Put $V=\O_{\theta}(\mathcal{A})$. Further, we will prove several auxiliary lemmas.

\begin{lemm}\label{orderthin}
In the above notations, $p^2\leq |V| \leq |G|/p^2$.
\end{lemm}

\begin{proof}
Lemma~\ref{psring}(1) yields that $|V|\geq p$. Clearly, $V\neq G$ because otherwise $\mathcal{A}=\mathbb{Z}G$. So $p\leq |V| \leq |G|/p$. To prove the lemma, we need to consider the cases $|V|=|G|/p$ and $|V|=p$. If $|V|=|G|/p$, then $\mathcal{A}$ is decomposable by Lemma~\ref{psring}(3), a contradiction to Corollary~\ref{wrnonnorm}.. 

Let $|V|=p$. Since $|G|\geq p^2$, there is an $\mathcal{A}$-subgroup $U\geq V$ of order~$p^2$ by Lemma~\ref{psring}(2). As $|V|=p$, we obtain $\mathcal{A}_U\cong \mathbb{Z}C_p\wr \mathbb{Z}C_p$ by Lemma~\ref{psring}(6). This implies that there is $X\in \mathcal{S}(\mathcal{A})$ which is a $V$-coset. Therefore $\mathcal{A}$ is decomposable by Lemma~\ref{notsizep}, a contradiction to Corollary~\ref{wrnonnorm}.
\end{proof}

\begin{lemm}\label{orderm}
In the above notations, $M$ is abelian and $|M|=|G|/|V|$.
\end{lemm}

\begin{proof}
One can see that $K=G_r\rtimes M$. So 
$$M\cong K/G_r=G_rR/G_r\cong R/(G_r\cap R)=R/V_r,$$
where the last equality follows from Lemma~\ref{intersection}. Together with $R\cong G$, this implies the required.
\end{proof}

\begin{lemm}\label{sizeset}
In the above notations, $p^2\leq |X| \leq |G|/|V|$ for every $X\in \mathcal{S}(\mathcal{A})_{G\setminus V}$.
\end{lemm}

\begin{proof}
Since $\mathcal{A}=\cyc(M,G)$, the basic set $X$ is an orbit of $M$. So $|X|\leq |M|=|G|/|V|$, where the last equality follows from Lemma~\ref{orderm}. Clearly, $|X|\geq p$. If $|X|=p$, then there is a basic set of $\mathcal{A}$ which is a coset by a subgroup of $V$ of order~$p$ by Lemma~\ref{sizep}. Therefore $\mathcal{A}$ is decomposable by Lemma~\ref{notsizep}, a contradiction to Corollary~\ref{wrnonnorm}. Thus, $|X|\geq p^2$.  
\end{proof}

\begin{lemm}\label{contthin0}
In the above notations, $|\O_{\theta}(\mathcal{A}_U)|\geq p^2$ for every $\mathcal{A}$-subgroup $U$ of order at least~$p^2$.
\end{lemm}

\begin{proof}
Note that $|\O_{\theta}(\mathcal{A}_U)|\geq p$ by Lemma~\ref{psring}(1). Assume that $|\O_{\theta}(\mathcal{A}_U)|=p$. Due to Lemma~\ref{psring}(2), $U$ has an $\mathcal{A}_U$-subgroup $L$ of order~$p^2$. Since $|\O_{\theta}(\mathcal{A}_U)|=p$, we conclude that $|\O_{\theta}(\mathcal{A}_L)|=p$. So $\mathcal{A}_L \cong \mathbb{Z}C_p\wr \mathbb{Z}C_p$ and hence $\mathcal{A}$ has a basic set which is a coset by a subgroup of $V$ of order~$p$. Therefore $\mathcal{A}$ is decomposable by Lemma~\ref{notsizep}, a contradiction to Corollary~\ref{wrnonnorm}.
\end{proof}

\begin{lemm}\label{contthin}
In the above notations, let $|V|=p^2$ and $X\in \mathcal{S}(\mathcal{A})_{G\setminus V}$ such that $|X|=|G|/p^2$. Then $\rad(X)\geq V$. 
\end{lemm}

\begin{proof}
Let $U=\langle X \rangle$. Due to Lemma~\ref{contthin0}, we have $U\geq V$. Together with the first part of Lemma~\ref{interrad} applied to $\mathcal{A}_U$ and inequality $|U|\geq p|G|/|V|$, this implies that
$$|\rad(X)\cap V|\geq p|X||V|/|U|=p|G|/|U|\in \{p,p^2\}.$$
If $|\rad(X)\cap V|=p$, then $\mathcal{A}$ is decomposable by the second part of Lemma~\ref{interrad}, a contradiction to Corollary~\ref{wrnonnorm}. Thus, $|\rad(X)\cap V|=p^2$, i.e. $\rad(X)\cap V=V$ as required.  
\end{proof}

Let us return to the proof of Theorem~\ref{main4}. If $G$ is elementary abelian and $n\leq 5$, then $G$ is $\CI^{(2)}$ by~\cite{FK} and hence $\NCI^{(2)}$, a contradiction to the assumption that $G$ is non-$\NCI^{(2)}$. So $G$ is non-elementary abelian. If $n\leq 3$, then $G$ is $\NCI^{(2)}$ by Theorem~\ref{main3}, again we obtain a contradiction to the assumption. Therefore $n=4$ or $n=5$.

Let $n=4$. In this case, $|V|=p^2$ by Lemma~\ref{orderthin} and $|X|=p^2$ for every $X\in \mathcal{S}(\mathcal{A})_{G\setminus V}$ by Lemma~\ref{sizeset}. Therefore every basic set from $\mathcal{S}(\mathcal{A})_{G\setminus V}$ is a $V$-coset by Lemma~\ref{contthin}. Thus, $\mathcal{A}=\mathcal{A}_V\wr \mathcal{A}_{G/V}$, a contradiction to Corollary~\ref{wrnonnorm}.

Let $n=5$. Then $|V|=p^2$ or $|V|=p^3$ by Lemma~\ref{orderthin}. Suppose that 
$$|V|=p^2.$$ 
In this case, $|M|=p^3$ by Lemma~\ref{orderm} and $|X|\in \{p^2,p^3\}$ for every $X\in \mathcal{S}(\mathcal{A})_{G\setminus V}$ by Lemma~\ref{sizeset}. Lemma~\ref{contthin} implies that $\rad(X)\geq V$ for every $X\in \mathcal{S}(\mathcal{A})$ of size~$p^3$. If this also holds for every $X\in \mathcal{S}(\mathcal{A})$ of size~$p^2$, then $\mathcal{A}=\mathcal{A}_V\wr \mathcal{A}_{G/V}$, a contradiction to Corollary~\ref{wrnonnorm}. Therefore there is $X\in \mathcal{S}(\mathcal{A})$ of size~$p^2$ such that 
$$\rad(X)\ngeq V.$$

By Lemma~\ref{contthin0}, we have $\langle X \rangle \geq V$. If $|\langle X \rangle|=p^3$, then $X$ is a $V$-coset by Lemma~\ref{psring}(4), a contradiction to $\rad(X)\ngeq V$. Therefore $|\langle X \rangle|\geq p^4$. Lemma~\ref{faith} yields that $X$ is a regular faithful orbit of $M$. Thus, $|X|=|M|=p^3$, a contradiction to $|X|=p^2$.

Now suppose that 
$$|V|=p^3.$$

\begin{lemm}\label{p2coset}
In the above notations, let $X\in \mathcal{S}(\mathcal{A})_{G\setminus V}$. Then $X$ is a coset by a subgroup of $V$ of order~$p^2$.
\end{lemm}

\begin{proof}
Lemma~\ref{sizeset} implies that $|X|=p^2$. Let $U=\langle X \rangle$. If $|U:\O_{\theta}(\mathcal{A}_U)|=p$, then we are done by Lemma~\ref{psring}(3). Further, we assume that $|U:\O_{\theta}(\mathcal{A}_U)|\geq p^2$. Together with Lemma~\ref{contthin0}, this implies that 
$$|U|\geq p^4,$$ 
$|\O_{\theta}(\mathcal{A}_U)|=p^2$ whenever $|U|=p^4$. Clearly, if $U=G$, then $|\O_{\theta}(\mathcal{A}_U)|=|V|=p^3$. So $|U:\O_{\theta}(\mathcal{A}_U)|=p^2$ in the both cases. From the first part of Lemma~\ref{interrad} applied to $\mathcal{A}_U$ it follows that $|\rad(X)\cap V|\geq p$. We are done if $|\rad(X)\cap V|=p^2$. Thus, we may assume that
$$|\rad(X)\cap V|=p.$$

If $|U|=p^5$, i.e $U=G$, then $\mathcal{A}$ is decomposable by the second part of Lemma~\ref{interrad}, a contradiction to Corollary~\ref{wrnonnorm}. Suppose that $|U|=p^4$. Then $|\O_{\theta}(\mathcal{A}_U)|=p^2$. The second part of Lemma~\ref{interrad} applied to $\mathcal{A}_U$ yields that $\mathcal{A}_U$ is the nontrivial $W/L$-wreath product, where $L=\rad(X)\cap V$ and $W$ is an $\mathcal{A}_U$-subgroup of index~$p$. Since $|V|=p^3$ and $|V\cap U|=|\O_{\theta}(\mathcal{A}_U)|=p^2$, we conclude that $G=UV$. Therefore $\mathcal{A}$ is the nontrivial $(WV)/L$-wreath product, a contradiction to Corollary~\ref{wrnonnorm}.  
\end{proof}

\begin{lemm}\label{stabautp2}
In the above notations, $|\aut(\mathcal{A})|=p^7$.
\end{lemm}

\begin{proof}
Clearly, $M\leq \aut(\mathcal{A})_e$. Lemma~\ref{orderm} implies that $|M|=p^2$. We are done if $\aut(\mathcal{A})_e=M$. Assume that $\aut(\mathcal{A})_e>M$, i.e. $|\aut(\mathcal{A})_e|\geq p^3$. Let $X \in \mathcal{S}(\mathcal{A})_{G\setminus V}$. Then $|X|=p^2$ by Lemma~\ref{p2coset}. One can see that given $x\in X$, the stabilizer $\aut(\mathcal{A})_{e,x}$ is nontrivial because $|\aut(\mathcal{A})_e|\geq p^3$ and $|X|=p^2$. Since $\mathcal{A}$ is normal, $\aut(\mathcal{A})_{e,x}\leq \aut(G)$. Put $\mathcal{B}=\cyc(\aut(\mathcal{A})_{e,x},G)$. On the one hand, $\mathcal{B}\neq \mathbb{Z}G$ because $\aut(\mathcal{A})_{e,x}$ is nontrivial. On the other hand, 
$$|\O_{\theta}(\mathcal{B})|\geq |\langle V,x\rangle|\geq p^4.$$
Therefore $|\O_{\theta}(\mathcal{B})|=p^4$. Thus, $\mathcal{B}$ is the nontrivial $(\langle V,x\rangle)/L$-wreath product for some $L\leq \langle V,x\rangle$ by Lemma~\ref{psring}(3). 

Every basic set from $\mathcal{S}(\mathcal{B})_{G\setminus V}$ is contained in a $V$-coset by Lemma~\ref{p2coset}. Since $\mathcal{B}\geq \mathcal{A}$, every basic set from $\mathcal{S}(\mathcal{B})_{G\setminus \langle V,x\rangle}$ so is and hence $L\leq V$. Observe that $\langle V,x\rangle=\langle V,X\rangle$ by Lemma~\ref{p2coset}. Therefore  $\langle V,x\rangle$ is an $\mathcal{A}$-subgroup as well as $\mathcal{B}$-subgroup. Thus, $\mathcal{A}$ is the $(\langle V,x\rangle)/L$-wreath product, a contradiction to Corollary~\ref{wrnonnorm}.
\end{proof}

If $V\cong C_{p^3}$ or $V\cong C_{p^2}\times C_p$, then every subgroup of $V$ of order~$p^2$ contains the characteristic subgroup~$L$ of $V$ of order~$p$. Together with Lemma~\ref{p2coset}, this implies that $\mathcal{A}$ is the $V/L$-wreath product, a contradiction to Corollary~\ref{wrnonnorm}. Therefore
$$V\cong C_p^3.$$

Assume that $G/V\cong C_{p^2}$. Let $X\in \mathcal{S}(\mathcal{A})_{G\setminus V}$ such that $X/V$ contains a generator of $G/V$. Due to Lemma~\ref{p2coset}, we have $X=Lx$ for some $x\in X$, where $L=\rad(X)\leq V$ is of order~$p^2$. One can see that $G/L=\langle V/L, X/L\rangle$. So $\mathcal{A}_{G/L}=\mathbb{Z}(G/L)$ by Lemma~\ref{groupring}. Together with Lemma~\ref{p2coset}, this yields that every basic set from $\mathcal{S}(\mathcal{A})_{G\setminus V}$ is an $L$-coset. Therefore $\mathcal{A}$ is the $V/L$-wreath product, a contradiction to Corollary~\ref{wrnonnorm}. Thus,
$$G/V\cong C_p^2.$$

Due to the above equality, one can choose $x,y\in G\setminus V$ such that 
$$G/V=\langle Vx \rangle \times \langle Vy \rangle.$$
Every element $g\in G$ can be uniquely presented in the form $g=x^iy^jv$, where $i,j\in \{0,\ldots,p-1\}$ and $v\in V$. Given $i,j\in \{0,\ldots,p-1\}$ such that $(i,j)\neq (0,0)$, the basic set of $\mathcal{A}$ containing $x^iy^j$ is denoted by $X_{ij}$. Lemma~\ref{p2coset} implies that $X_{ij}=L_{ij}x^iy^j$ for some $L_{ij}\leq V$ of order~$p^2$. Clearly, $\rad(X_{ij}v)=L_{ij}$ for every $v\in V$.

\begin{lemm}\label{p2coset1}
In the above notations, let $i,j,i^\prime,j^\prime\in \{0,\ldots,p-1\}$ such that $(i,j)\neq (0,0)$ and $(i^\prime,j^\prime)\neq (0,0)$. Then $L_{ij}=L_{i^\prime j^\prime}$ if and only if $\langle V,x^iy^j\rangle=\langle V,x^{i^\prime}y^{j^\prime}\rangle$.
\end{lemm}

\begin{proof}
If $\langle V,x^iy^j\rangle=\langle V,x^{i^\prime}y^{j^\prime}\rangle$, then $X_{i^\prime j^\prime}$ can be obtained from $X_{ij}$ by rational conjugation (see Lemma~\ref{burn}) which preserves a radical of a basic set.

Suppose that $\langle V,x^iy^j\rangle\neq\langle V,x^{i^\prime}y^{j^\prime}\rangle$ and $L_{ij}=L_{i^\prime j^\prime}:=L$. Then $\langle x^iy^j,x^{i^\prime}y^{j^\prime},V\rangle=G$, $X_{ij}=Lx^iy^j$, and $X_{i^\prime j^\prime}=Lx^{i^\prime}y^{j^\prime}$. So the basic sets $X_{ij}/L$ and $Y_{i^\prime j^\prime}/L$ of $\mathcal{A}_{G/L}$ are singletons. Therefore $\mathcal{A}_{G/L}=\mathbb{Z}(G/L)$ by Lemma~\ref{groupring}. This yields that every basic set of $\mathcal{A}$ is contained in an $L$-coset. Due to Lemma~\ref{p2coset}, we conclude that every basic set from $\mathcal{S}(\mathcal{A})_{G\setminus V}$ is an $L$-coset. Thus, $\mathcal{A}$ is the $V/L$-wreath product, a contradiction to Corollary~\ref{wrnonnorm}.  
\end{proof}

Let $\widetilde{G}=V\times \langle \widetilde{x}\rangle \times \langle \widetilde{y}\rangle$, where $|\widetilde{x}|=|\widetilde{y}|=p$. Since $V\cong C_p^3$, we have $\widetilde{G}\cong C_p^5$. Every element $\widetilde{g}$ of $\widetilde{G}$ can be uniquely presented in the form $\widetilde{g}=\widetilde{x}^i\widetilde{y}^jv$, where $i,j\in \{0,\ldots,p-1\}$ and $v\in V$. Let us consider a partition $\widetilde{\mathcal{S}}$ of $\widetilde{G}$ into the following sets:
$$\{v\},~\widetilde{X}_{ij}=L_{ij}\widetilde{x}^i\widetilde{y}^jw_{ij},$$
where $v\in V$, $i,j\in \{0,\ldots,p-1\}$ are such that $(i,j)\neq 0$, and $w_{ij}$ runs over a transversal for $L_{ij}$ in $V$. Clearly, $\widetilde{X}_{ij}^{-1}=\widetilde{X}_{p-i,p-j}$. Using Lemma~\ref{p2coset1}, one can compute that
$$\underline{\widetilde{X}_{ij}}\cdot \underline{\widetilde{X}_{i^\prime j^\prime}}=
\begin{cases}
p\underline{V\widetilde{x}^{i+i^\prime}\widetilde{y}^{j+j^\prime}},~\langle x^iy^j\rangle\neq\langle x^{i^\prime}y^{j^\prime}\rangle,\\
p^2\underline{L_{ij}\widetilde{x}^{i+i^\prime}\widetilde{y}^{j+j^\prime}},~\langle x^iy^j\rangle=\langle x^{i^\prime}y^{j^\prime}\rangle.\\
\end{cases}$$
Therefore $\widetilde{\mathcal{A}}=\Span_{\mathbb{Z}}\{\underline{\widetilde{X}}:~\widetilde{X}\in \widetilde{\mathcal{S}}\}$ is an $S$-ring over $\widetilde{G}$. 

Let us define a bijection $f$ from $G$ to $\widetilde{G}$ as follows:
$$(x^iy^jv)^f=\widetilde{x}^i\widetilde{y}^jv$$
for all $i,j\in \{0,\ldots,p-1\}$ and $v\in V$. Let us check that $f$ is an isomorphism from $\mathcal{A}$ to $\widetilde{\mathcal{A}}$. Due to the definition of $f$, we conclude that $v^f=v$ and 
$$(X_{ij}v)^f=(L_{ij}vx^iy^j)^f=L_{ij}v\widetilde{x}^i\widetilde{y}^j=\widetilde{X}_{ij}v$$
for all $i,j\in \{0,\ldots,p-1\}$ and $v\in V$. Thus, $f$ maps the basic sets of $\mathcal{A}$ to the basic sets of $\widetilde{\mathcal{A}}$.

Let $g,g^\prime\in G$. Then $g=x^iy^jv$ and $g^\prime=x^{i^\prime}y^{j^\prime}v^\prime$ for some $i,j,i^\prime,j^\prime\in \{0,\ldots,p-1\}$ and $v,v^\prime\in V$ and hence 
$$g^\prime g^{-1}=x^{i^\prime-i}y^{j^\prime-j}v^{\prime}v^{-1}.$$ 
The definition of $f$ implies that $g^f=\widetilde{x}^i\widetilde{y}^jv$ and $(g^\prime)^f=\widetilde{x}^{i^\prime}\widetilde{y}^{j^\prime}v^\prime$ and consequently 
$$(g^\prime)^f(g^f)^{-1}=\widetilde{x}^{i^\prime-i}\widetilde{y}^{j^\prime-j}v^{\prime}v^{-1}.$$ 
Let $X\in \mathcal{S}(\mathcal{A})$ containing $x^{i^\prime-i}y^{j^\prime-j}v^{\prime}v^{-1}$. Then $X^f\in \mathcal{S}(\widetilde{A})$ contains $\widetilde{x}^{i^\prime-i}\widetilde{y}^{j^\prime-j}v^{\prime}v^{-1}$. Therefore
$$g^\prime g^{-1}\in X\Leftrightarrow (g^\prime)^f(g^f)^{-1}\in X^f.$$
Thus, $f$ is an isomorphism from $\mathcal{A}$ to $\widetilde{\mathcal{A}}$. Together with Lemma~\ref{stabautp2}, this implies that 
\begin{equation}\label{auttilde}
|\aut(\widetilde{\mathcal{A}})|=p^7.
\end{equation}

Let $\widetilde{N}=N_{\aut(\widetilde{\mathcal{A}})}(\widetilde{G}_r)\leq \Hol(\widetilde{G})$. Assume that $|\widetilde{N}_e|=p$. Put $\widetilde{\mathcal{B}}=\cyc(\widetilde{N}_e,\widetilde{G})$. One can see that $\widetilde{N}\leq \aut(\widetilde{\mathcal{A}})$ and hence $\widetilde{\mathcal{B}}\geq \widetilde{\mathcal{A}}$. Clearly, every basic set of $\widetilde{\mathcal{B}}$ is of size at most~$p$. Therefore $\widetilde{\mathcal{B}}>\widetilde{\mathcal{A}}$. This implies that $\aut(\widetilde{\mathcal{B}})<\aut(\widetilde{\mathcal{A}})$. On the other hand, $\aut(\widetilde{\mathcal{B}})\geq \widetilde{N}$ and consequently 
$$\aut(\widetilde{\mathcal{B}})=\widetilde{N}.$$

Let $f\in \aut(\widetilde{\mathcal{A}})\setminus \widetilde{N}$. Then $\widetilde{G}_r^f\neq \widetilde{G}_r$. As $|\aut(\widetilde{\mathcal{A}}):\widetilde{N}|=p$, we conclude that $\widetilde{N}$ is normal in $\aut(\widetilde{\mathcal{A}})$ and hence  $\widetilde{G}_r^f\leq \widetilde{N}$. Thus, $\widetilde{G}_r^f$ and $\widetilde{G}_r$ are nonconjugate $\widetilde{G}$-regular subgroups of $\widetilde{N}=\aut(\widetilde{\mathcal{B}})$. This yields that $\widetilde{\mathcal{B}}$ is a non-$\CI$-$S$-ring over $\widetilde{G}$. However, $\widetilde{G}\cong C_p^5$ is a $\CI^{(2)}$-group by~\cite{FK} which means that $\widetilde{\mathcal{B}}$ is $\CI$, a contradiction.

The above discussion implies that 
$$|\widetilde{N}_e|=p^2$$
and hence $\widetilde{\mathcal{A}}=\cyc(\widetilde{N}_e,\widetilde{G})$. Note that $\widetilde{\mathcal{A}}$ is indecomposable by Corollary~\ref{wrnonnorm} and $|\O_{\theta}(\widetilde{\mathcal{A}})|=|V|=p^3$ because $\widetilde{\mathcal{A}}$ and $\mathcal{A}$ are isomorphic. Therefore $\widetilde{\mathcal{A}}$ satisfies all the conditions of~\cite[Lemma~6.4]{FK}. From~\cite[Lemma~6.5]{FK} it follows that $|\aut(\widetilde{\mathcal{A}})|=p^8$, a contradiction to Eq.~\eqref{auttilde}.

\end{document}